%% file: main.tex
\documentclass[11pt,reqno]{amsart}
\usepackage[margin=1in]{geometry}
\usepackage{setspace}
\usepackage{datetime}
\usepackage{mysty}
\usepackage{biblatex}
\usepackage{thmtools}
\usepackage{thm-restate}
\usepackage{svg}

\definecolor{wine}{RGB}{200,60,100}

\title{A Concordance Invariant For Knots from $CFK^\infty$}
\author{Kashti Satish Umare}
\date{\today}
\begin{document}
    \vspace*{-1cm}
    \maketitle
    \setlength{\parindent}{0pt}
    \singlespacing
    \hypersetup{citecolor=wine,linkcolor=wine,urlcolor=wine}
    \tableofcontents

\input{mainfile}

    \newpage
    \printbibliography
\end{document}

%% file: mainfile.tex
\ifdefined\thm\else
  \newtheorem{thm}[thm]{Theorem 2}
\fi

\newtheorem{lemma}{Lemma}
\newtheorem{corollary}{Corollary}
\newtheoremstyle{named}{}{}{\itshape}{}{\bfseries}{.}{.5em}{\thmnote{#3 }#1}

\begin{abstract}
We construct a multi-filtered smooth knot concordance invariant, $\beta(t^1, a, t^2)$, derived from the knot Floer complex $CFK^\infty$. $\beta$ is a three-variable piecewise linear function defined on $(0,2) \times \mathbb{R} \times (0,2)$. We demonstrate that $\beta$ satisfies a subadditivity relation. We apply $\beta$ to prove concordance results for a certain family of L-space knots. We show that the positive cones spanned by Teragaito's knots $K_{1,n}$ and $K_{2,n}$ are disjoint in the smooth concordance group $\mathcal{C}$. For positive L-space knots, we show $\beta$ completely determines the Alexander polynomial and distinguishes knots with identical $\Upsilon$ invariants. 
\end{abstract}

\section{Introduction}

This paper aims to construct a smooth concordance invariant $\beta(t^1,a,t^2)$, which is a function of three variables, $t^1,t^2 \in(0,2)$ and $a \in \mathbb{R}$, and is piecewise linear in each variable. This work is inspired by Kim and Livingston's "secondary upsilon invariant" $\Upsilon^2$ \cite{KimLivingston}, which is a set of piecewise linear functions on the interval $(0,2)$. The number of piecewise linear functions that describe $\Upsilon^2$ is based on the singularities of Ozsváth, Stipsicz, and Szabó's "upsilon invariant" $\Upsilon$, which is a homomorphism from the space of knots to the set of piecewise linear functions on the interval $[0,2]$ \cite{upsilonozsvath}.

 As an application of $\beta$, we use it to prove the following result about the smooth concordance group $\mathcal{C}$ \cite{livingston2004surveyclassicalknotconcordance}. Recall Teragaito \cite{Teragaito2025Hyperbolic} constructs an infinite family of hyperbolic L-space knots which share the same $\Upsilon$ but have different Alexander polynomials. They are defined using the diagram in Figure \ref{fig:teresknots}. 

\begin{figure}
    \centering
    \includegraphics[width=0.7\linewidth]{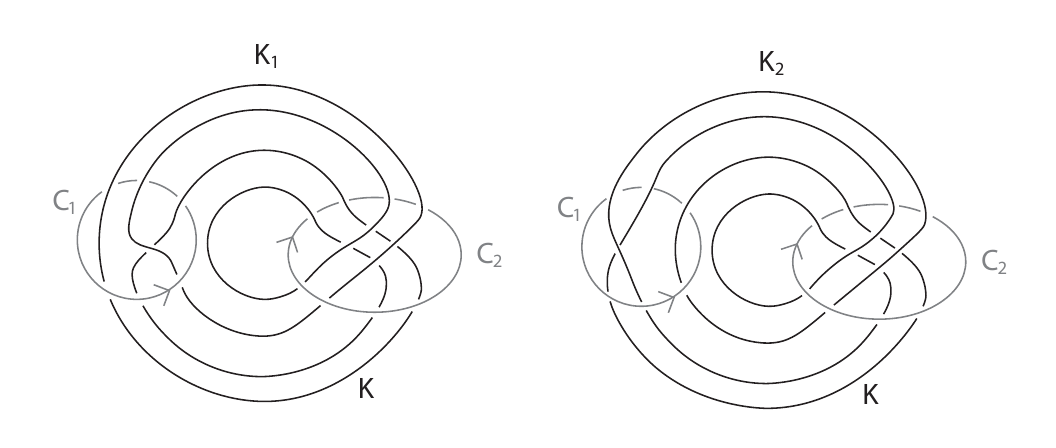}
    \caption{ The knots $K_1$ and $K_2$ are images of $K$ after performing $(-1/n)$ surgery on $C_2$ and $(-1/2)$ surgery on $C_1$. $K_{1,n}$ refers to the knots of the type on the left and $K_{2,n}$ refers to knots of the type on the right.}
    \label{fig:teresknots}
\end{figure}

\begin{restatable}{thm}{teragaitoCone}
\label{thm:teragaito}
    The positive cone in $\mathcal{C}$ spanned by $K_{1,n}$ is disjoint from the positive cone spanned by $K_{2,n}.$ That is, for any knots of the type $K_{1,n}$ and $K_{2,n}$, there is no nontrivial expression of the form $$\#_{n}c_n K_{1,n} \sim \#_{m}d_mK_{2,m}$$ for $c_n,d_m \in \mathbb{Z}_{\geq0}$, where $\sim$ denotes knot concordance. 
\end{restatable}

\textbf{Acknowledgments:}
I would like to thank Peter Ozsváth for his support and mentorship. I would also like to thank Jen Hom for valuable comments. 

\textbf{AI Declaration:}
AI was not used to generate any of the mathematical content in this paper. AI was used for minor grammatical and typesetting assistance.

%\begin{restatable}{thm}{pretzelSum}
%\label{thm:pretzel}
%We have from \cite{lidman2016pretzel}, that pretzel knots of the form $P(-2,3,n)$ for $n\geq 7$ odd are L-space knots. If $\#_n c_nK_n \sim \#_m d_mJ_m$ where $K_n$ and $J_m$ are pretzel knots, then $\sum_n c_n = \sum_n d_n$ for $c_n,d_n \geq 0$.
%\end{restatable} This result is inspired by \cite{Alfieri2019Upsilon}

\section{Background}

The knot Floer complex $CFK^-(K)$ is a free module over $\mathbb{F}[U]$, where $\mathbb{F}$ is the field with two elements (this is the complex $CFK^{-,*}$ from \cite{ozsvath2004holomorphic} and this is $\mathcal{GC}^{-}$ from \cite{GridHom}). This complex is equipped with a Maslov grading $M$ and an Alexander filtration $A$. The Maslov grading $M$ is a homological grading that arises from the Maslov index of holomorphic discs, and the Alexander filtration tracks how a generator interacts with a Seifert surface of $K$. For more information on these definitions, see \cite{ozsvath2017overviewknotfloerhomology}. Multiplication by $U$ performs the shifts $M(U\cdot x) = M(x)-2$ and $A(U \cdot x) = A(x)-1$. The differential of the chain complex satisfies $M(\partial x) = M(x)-1$ and $A (\partial x) \leq A(x)$. The filtered chain homotopy type of $CFK^-$ is an invariant of the knot $K$. 

We next define the $\mathbb{F}[U]$-module $HFK^-(K)$ as the homology of the associated graded object of $CFK^-(K)$ with respect to the Alexander filtration. Concretely, we get $$HFK^-(K) = \bigoplus_{j \in \mathbb{Z}} HFK^-(K,j), $$ where $HFK^-(K,j)$ is the homology of the associated graded piece at Alexander filtration level $j$. An element $x \in HFK^-(K)$ is called torsion if $U^n \cdot x =0$ for some $n \geq 1$ and non-torsion if $U^n \cdot x \neq0$ for all $n \geq 1$.

$CFK^\infty(K)$ is a $\mathbb{F}[U, U^{-1}]$-module given by formally inverting $U$ and adding an algebraic filtration to $CFK^-(K)$. The algebraic filtration is given by $\mathcal{F}_{\text{alg}}^i = \{ U^n x | n \geq -i \}$, where $x$ is in the basis, up to indexing conventions. Thus $F_{\text{alg}}^0 = CFK^- (K)$.

%To each $CFK^\infty(K)$ we can associate a finite dimensional complex $C_0$ such that $CFK^\infty (K) = C_0 \otimes \mathbb{F}[U, U^{-1}]$.  

Graphically, we represent $CFK^\infty$ as follows: It is drawn on a grid using vertices and arrows. Each vertex represents a generator (over $\mathbb{F}{[U,U^{-1}]}$) of $CFK^\infty (K)$. The coordinate ${i,j}$ represents the algebraic and Alexander filtration levels of the generator. Arrows represent components of the boundary. An example of the $CFK^\infty$ diagram for $T_{2,5}$ is provided in Figure \ref{fig:2,5}. The filled dots represent vertices in Maslov grading 0, called cycles, and the unfilled vertices are in Maslov grading 1. 

\begin{figure}
    \centering
    \includegraphics[width=0.35\linewidth]{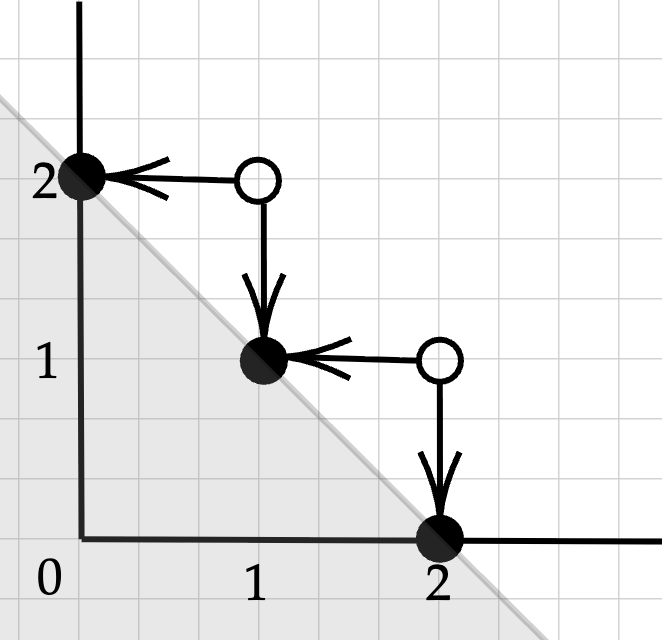}
    \caption{$CFK^\infty (T_{2,5})$}
    \label{fig:2,5}
\end{figure}

%To consider the full complex, we must take all integer-valued diagonal translations of the complex. To do this, we note that the action of $U$ shifts the vertices one unit down and one to the left. Thus, we can conceive of the diagram as a finite-dimensional illustration of the complex $C_0$, which is finite-dimensional and $CFK^\infty (K) = C_0 \otimes \mathbb{F}[U, U^{-1}]$ graded and filtered so that action by $U$ lowers gradings by 2 and filtration levels by 1.

%The minimal algebraic filtration level of representatives of the nontrivial homology class is 0, and likewise for the Alexander filtration level. Therefore if our schematic is connected, the schematic will determine the grading level of all the vertices. While the homology given by $H_*(C)$ can capture data such as the $\tau$ invariant, we attempt to find more differentiating information.

We now give background on the Ozsváth-Szabó $\Upsilon$ \cite{upsilonozsvath} and the Kim-Livingston $\Upsilon^2$ \cite{KimLivingston} invariants. In \cite{KimLivingston}, the authors construct the secondary upsilon invariant $\Upsilon^2$, which is a concordance invariant. This extends the paper \cite{upsilonozsvath}, which constructs a concordance homomorphism $\Upsilon$.

For $t \in (0,2)$ we define the filtration $F_t (i,j) = (1-t/2)i+(t/2)j$. Here $i$ represents the algebraic filtration level, and $j$ represents the Alexander filtration level. We define $F_{s,t}(C)$ to be the set of all elements with filtration levels less than or equal to $s$ where $C$ is a bifiltered chain complex.  When we have a knot $K$, we write $F_{s,t}(K)$ to mean $F_{s,t}(CFK^\infty (K))$.

First, we define $\Upsilon$. Let $$\gamma_K(t) = \min\{s| H_0(F_{s,t}) \to H_0(CFK^\infty (K)) \cong \mathbb{F} \text{ is surjective} \}.$$ Then $\Upsilon_K(t) = -2 \gamma_K(t)$. \cite{upsilonozsvath} shows that $\Upsilon$ provides a homomorphism from the concordance group to the space of piecewise linear functions on the interval $[0,2]$. 

To define $\Upsilon^2$, we first define the support line $\mathcal{L}_{s,t}$.  Note that $F_{s,t}$ is represented by a half-space which has a boundary line with slope $m = 1-2/t$. This is demonstrated by the gray diagonal line in Figure \ref{fig:2,5}. Given the bounding line, define $s = tj_0/2$, where $j_0$ is the j-intercept of the line. This line is $\mathcal{L}_{s,t}$. 

Now let $\mathcal{P}$ be the set of bifiltration levels of elements of $CFK^\infty(K)$. Note that for $t\in (0,2)$, $2 \gamma_K(t) = -\Upsilon_K(t)$. Thus the line $\mathcal{L}_{\gamma(t),t}$ contains a nonempty subset of $\mathcal{P}$, which we call $\mathcal{P}(t)$. For small $\delta$, $\mathcal{P}_{t-\delta}$ has exactly one element of $\mathcal{P}_t$, and likewise for $\mathcal{P}_{t+\delta}$. We call these the negative and positive pivot points $p^-_t$ and $p^+_t$. 

Now let $t^\pm = t\pm \delta$. We let $\mathcal{Z}^\pm = \{ z_j^\pm\}$ be the set of cycles in $F_{\gamma_K(t^\pm),t^\pm}$ which represent nontrivial elements of $H_0(CFK^\infty(K))$. For all $s \in (0,2)$, define $\gamma^2_{K,t}(s)$ to be the minimum value $r$ such that some $z_i^-$ and $z_i^+$ represent the homology class of $H_0(F_{t,\gamma_k(t)}+ F_{s,r})$. Finally, $\Upsilon^2_{K,t}(s) = -2\gamma^2_{K,t}(s) - \Upsilon_{K,t} = -2(\gamma^2_{K,t}(s) -\gamma_K(t)) $. While $\Upsilon^2$ is not a homomorphism, it is in fact a concordance invariant, as proven in \cite{KimLivingston}.

As an example, we compute upsilon and secondary upsilon for $T_{2,5}$, whose $CFK^\infty$ schematic is given in Figure \ref{fig:2,5}. Here, the filtration vertex which generates the minimal $s$ from time $0$ to $1$ is $(0,2)$, and from $1$ to $2$ it is $(2,0)$. Thus 
\[
\gamma(t) = \begin{cases} 
      t & 0\leq t \leq 1 \\
2-t &  1< t \leq 2
   \end{cases}
\]

which implies 

\[
\Upsilon(t) = \begin{cases} 
      -2t & 0\leq t \leq 1 \\
-4+2t &  1< t \leq 2
   \end{cases}
\]

For $\Upsilon^2$, we simply need to compute $\gamma^2_{K,t}(s)$. Here, the point where the two vertices represent the same homology class is $t=1$. Thus we compute $\gamma^2_{K,1}(s)$.  We compute that for $s \in (0,2)$ the line through (1,2) and $(2,1)$ gives $$r = (s/2)(2) + (1-s/2)(1) = (s/2)(1) + (1-s/2)(2)  = 1+s/2 = \gamma_{K,1}^2(s) $$

at $t=1$, $\gamma = 1$.  Thus $$\Upsilon^2_{k,1}(s) = -2(1+s/2-1) = -s.$$

\section{Definition and Properties of $\beta$}

\begin{defn}
    Fix real numbers $t \in (0,2)$ and $s \in \mathbb{R}$. Let $F_{s,t}(K) \subset CFK^\infty (K)$ be the subcomplex spanned by generators $U^i$ satisfying $A(U^i) = j$ and $(1-\frac{t}{2})\cdot i + \frac{t}{2}\cdot j \leq s$. Then for a real number $a$ and two directional parameters $t^1,t^2$, we define the multi-filtered invariant $$\beta_K(t^1,a,t^2) = \min\{s^2 \in \mathbb{R}| \iota_*:  H_0(F_{s^2+a,t^1}(K) \cap F_{s^2,t^2}(K)) \to H_0(CFK^\infty (K)) \cong \mathbb{F} \text{ is surjective} \}$$
\end{defn}

We note that $\beta$ is related to Alfieri's $\Upsilon^C$ \cite{Alfieri2019Upsilon}. 

First, we must prove that $\beta$ is in fact a concordance invariant. 

\begin{lemma} \label{thm:concordance invaraint}
    $\beta$ is a smooth concordance invariant.
\end{lemma}
\begin{proof} 

We proceed using Theorem 1 from \cite{hom2017survey}. Hom's Theorem 1 states that if two knots $K_1$ and $K_2$ are smoothly concordant, then their knot Floer complexes must be related by a filtered chain homotopy equivalence modulo acyclic summands $A_1$ and $A_2$. That is $$CFK^\infty(K_1) \oplus A_1 \simeq CFK^\infty(K_2) \oplus A_2.$$ 

\begin{figure}
    \centering
    \includegraphics[width=0.35\linewidth]{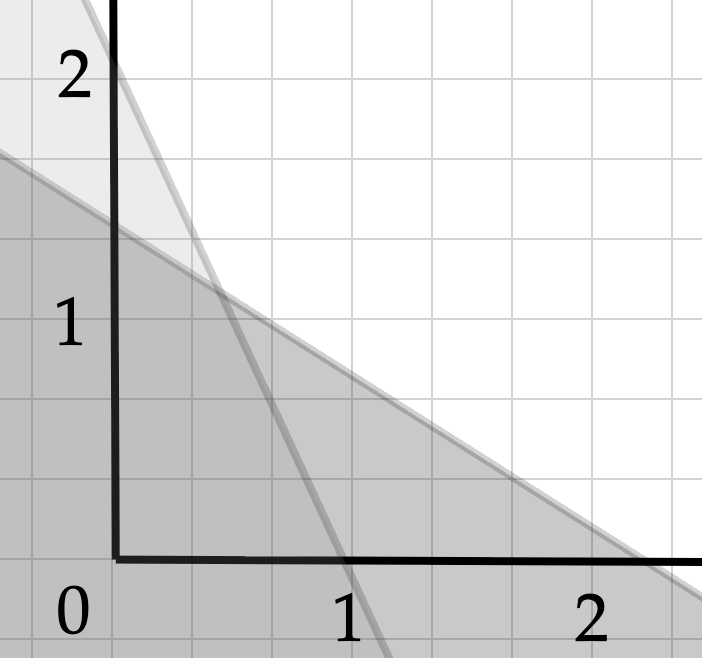}
    \caption{$D_C(s)$}
    \label{fig:example}
\end{figure}

For a bifiltered complex $C$, we fix $a,s,t^1,$ and $t^2$ and set $$D_C(s) = F_{s+a,t^1}(C) \cap F_{s,t^2}(C).$$ This is shown in \ref{fig:example}, where each line and the half plane it bounds corresponds to one of the filtration thresholds and we consider their intersection. Now suppose $C \simeq C'$ is a bifiltered chain homotopy equivalence. Any filtered chain map $\phi: C \to C'$ must satisfy $\phi(D_C(s)) \subseteq D_{C'}(s)$ since $\phi$ does not increase either of the two filtrations and thus does not increase $f_t(x) = (1-\frac{t}{2})i(x) + \frac{t}{2}j(x)$ for any $t \in (0,2)$. This means $\phi$ induces a commutative diagram on $H_0$. 

\[\begin{tikzcd}
	{H_0(D_C(s))} && {H_0(C)} \\
	\\
	{H_0(D_{C'}(s))} && {H_0(C')}
	\arrow[from=1-1, to=1-3]
	\arrow[from=1-1, to=3-1]
	\arrow[from=1-3, to=3-3]
	\arrow[from=3-1, to=3-3]
\end{tikzcd}\]

Because $\phi_*:H_0(C) \to H_0(C')$ is an isomorphism, the map on homology induced by inclusion is $H_0(D_C(s)) \to H_0(C)$ and it is surjective if and only if $H_0(D_{C'}(s)) \to H_0(C')$ is also surjective. Thus we have $$\beta(t^1,a,t^2) = \beta_{C'}(t^1,a,t^2).$$ 

Now suppose $A$ is an acyclic chain complex. That means $$D_{C \oplus A}(s) = D_C(s) \oplus D_A(s)$$ with $H_0(C \oplus A) \cong H_0(C)$ because $H_0(A) =0$. This means the map induced by the inclusion map for $C \oplus A$ is surjective exactly when the map which corresponds to $C$ is surjective. Thus $\beta(C \oplus A) = \beta(C)$. Then by Theorem 1 of Hom, $\beta_{K_1} = \beta_{K_1 \oplus {A_1}}=\beta_{K_2 \oplus {A_2}} = \beta_{K_2}$.

\end{proof}

\begin{lemma}
    If $X \sim 0$ in $C$ meaning that $X$ is smoothly slice, then $\beta_X(t^1,a,t^2) = 0$ for all $t^1,t^2 \in (0,2)$, $a \geq 0$.
\end{lemma}
\begin{proof}
    By Lemma \ref{thm:concordance invaraint}, $\beta_X(t^1,a,t^2) = \beta_U(t^1,a,t^2)$ where $U$ is the unknot. The reduced knot Floer complex for the unknot has a single generator $x_0$ in Maslov grading 0 supported at the origin $(i,j) = (0,0)$ with vanishing differential $\partial = 0$. For any directional parameter $t \in (0,2)$, the regional filtration of the generator $x_0$ is $$f_t(x_0) = (1-\frac{t}{2})\cdot 0 + \frac{t}{2}\cdot 0 = 0$$ meaning $x_0 \in F_{s,t}(U)$ if and only if $0 \leq s$. 

    By definition 3.1, $s^2$ is the minimal threshold such that $[x_0]$ lies in the image of the inclusion map $$\iota_* : H_0\left(F_{s^2+a, t^1}(U) \cap F_{s^2, t^2}(U)\right) \to H_0\left(CFK^\infty(U)\right).$$ This means $x_0 \in F_{s^2+a, t^1}(U) \cap F_{s^2, t^2}(U)$ so $$f_{t^1}(x_0) \leq s^2 + a \implies 0 \le s^2 + a \implies s^2 \geq -a$$ and $$f_{t^2}(x_0) \leq s^2 \implies 0 \leq s^2.$$

    By combining both inequalities we get $s^2 \geq \max(0,-a)$. Thus for any $a \geq 0,$ the constraint $s^2 \geq 0$ dominates, meaning that $\beta_X(t^1,a,t^2) = 0$ for all $t^1,t^2 > 0$.
\end{proof}

Next, we prove a subadditivity property of $\beta$.

\begin{lemma}
    For two knots $K_1,K_2$ given any $t^1,t^2 
    \in (0,2)$ and shifts $a_1,a_2 \in \mathbb{R}$, $\beta$ satisfies $$\beta_{K_1 \# K_2} (t^1,a_1+a_2,t^2) \leq \beta_{K_1}(t^1,a_1,t^2) + \beta_{K_2}(t^1,a_2,t^2)$$
\end{lemma}

\begin{proof}
    The full knot Floer complex of a connected sum is filtered chain homotopy equivalent to the tensor product \cite{ozsvath2004holomorphic} given by $$CFK^\infty (K_1 \# K_2) \cong CFK^\infty(K_1) \otimes_{\mathbb{F}[U,U^{-1}]}CFK^\infty (K_2).$$

    Under this tensor product, the two filtrations are additive on elementary tensors, meaning $(i_1,j_1) \otimes (i_2,j_2)\mapsto (i_1+i_2,j_1+j_2)$. Thus for $t \in (0,2)$ we have $f_t(x_1\otimes x_2) = f_t(x_1) + f_t(x_2)$. Now let us set $\beta_1 = \beta_{K_1} (t^1,a_1,t^2)$ and $\beta_2 = \beta_{K_2} (t^1,a_2,t^2)$. Then by our definition of $\beta$, there exist cycles $z_1 \in CFK^\infty (K_1)$ and $z_2 \in CFK^\infty (K_2)$ which represent the nonzero classes in $H_0(CFK^\infty (K_1))$ and $H_0(CFK^\infty (K_2))$ respectively. This gives us $$z_1 \in F_{\beta_1+a_1,t^1}(K_1) \cap F_{\beta_1,t^2}(K_1)$$ and $$z_2 \in F_{\beta_2+a_2,t^1}(K_2) \cap F_{\beta_2,t^2}.$$

    All elementary tensors in $z_1 \otimes z_2$ have at most $t^1$ filtration $$(\beta_1+a_1)+(\beta_2+a_2) = (\beta_1+\beta_2) + (a_1+a_2)$$ and at most $t_2$ filtration $$\beta_1+\beta_2$$ since $f_t(x_1\otimes x_2) = f_t(x_1) + f_t(x_2)$. Thus we obtain $$z_1 \otimes z_2 \in F_{(\beta_1+\beta_2) + (a_1+a_2),t^1}(K_1 \# K_2)\cap F_{\beta_1+\beta_2}(K_1 \# K_2).$$ Because we have that $[z_1]$ and $[z_2]$ are nonzero in the one dimensional vector spaces $H_0(CFK^\infty(K_1)) \cong \mathbb{F}$ and $H_0(CFK^\infty (K_2)) \cong \mathbb{F}$, their tensor $[z_1] \otimes [z_2]$ is nonzero and represents the generator of $$H_0(CFK^\infty (K_1 \# K_2)) \cong H_0(CFK^\infty(K_1)) \otimes_{\mathbb{F}} H_0(CFK^\infty (K_2)).$$ Thus we have that the inclusion $$H_0(F_{(\beta_1+\beta_2) + (a_1+a_2),t^1}\cap F_{\beta_1+\beta_2,t^2}) \to H_0CFK^\infty (K_1 \# K_2)$$ is surjective. By the definition of $\beta$ we get $$\beta_{K_1 \# K_2}(t^1,a_1+a_2,t^2) \leq \beta_1 + \beta_2.$$ 
    %For any threshold parameters $S$ and directional parameters $t$, we have $$F_{S,t}(K_1 \#K_2) = \bigcup_{s_1+s_2=S}F_{s_1,t}(K_1) \otimes_{\mathbb{F}[U,U^{-1}]}F_{s_2,t}(K_2).$$ Allow $s_1^2 = \beta_{K_1}(t^1,a_1,t^2) $ and $s^2_2 = \beta_{K_2}(t^1,a_2,t^2)$. By definition of the minimum threshold for surjectivity, there exist homology-generating cycles $z_1 \in CFK^\infty (K_1)$ and $z_2 \in CFK^\infty (K_2)$ so $$z_1 \in F_{s_1^2 +a_1,t^1}(K_1) \cap F_{s_1^2,t^2}(K_1) \text{   and   }z_2 \in F_{s_2^2 +a_2,t^1}(K_2) \cap F_{s_2^2,t^2}(K_2). $$

    %Now let us consider an elementary tensor cycle $z_1 \otimes z_2 \in CFK^\infty(K_1\#K_2)$. Then we can distribute regional intersection filters across tensor complexes so $$z_1 \otimes z_2 \in \big(F_{s_1^2 +a_1,t^1}(K_1) \otimes F_{s_2^2+a_2,t^1}(K_2)\big) \cap \big(F_{s_1^2,t^2}(K_1) \otimes F_{s_2^2,t^2}(K_2)\big) . $$ We now sum the corresponding filtration limits to see that $$z_1 \otimes z_2 \in F_{(s_1^2+s_2^2) + (a_1+a_2),t^1}(K_1\#K_2)\cap F_{s_1^2+s_2^2,t^2}(K_1 \# K_2).$$ Because $z_1$ and $z_2$ project to non-torsion $H_0(CFK^\infty(K_1))$ and $H_0(CFK^\infty(K_2))$, their tensor $z_1 \otimes z_2$ generates $H_0(CFK^\infty(K_1 \# K_2))$. Thus the true minimal parameter threshold $\beta_{K_1 \#K_2}(t^1,a_1+a_2,t^2)$ must be bounded above $$\beta_{K_1 \#K_2}(t^1,a_1+a_2,t^2)\leq s_1^2 +s_2^2 = \beta_{K_1}(t^1,a_1,t^2) + \beta_{K_2}(t^1,a_2,t^2).$$
    
\end{proof}

\section{${\beta}$ for L-space knots}

\begin{defn}
    Given a sequence $$S(K) = \{ (i_1,j_1),(i_2,j_2),...,(i_k,j_k) \}\subset \mathbb{Z}_{\geq 0} \times \mathbb{Z}_{\geq 0}$$ with $0=i_1 < i_2 < ... < i_k$ and $j_1 > j_2 >...>j_k=0$, consider the corresponding sequence $$R(K) = \{(i_{n+1},j_{n})| 1 \leq n \leq k-1\},$$ the positive staircase complex is the filtered complex
    \[ \bigoplus_{p \in S(K) \sqcup R(K)} p \cdot \mathbb F[U, U^{-1}] \]

    with a differential $\partial$ defined by $\partial p = 0$ if $p \in S(K)$ and $\partial (i_{n+1},j_{n}) = (i_n,j_n) + (i_{n+1},j_{n+1})$ for $(i_{n+1},j_{n}) \in R(K)$. The two filtrations are given by a projection onto the $i$ and $j$ axes, respectively. The vertices in $S(K)$ have Maslov grading 0 and the vertices in $R(K)$ have Maslov grading 1. An L-space knot is a knot $K$ such that $CFK^{\infty}(K)$ is filtered chain homotopic to a positive staircase complex.

\end{defn}

For a positive staircase complex, it is clear that the cycles in dimension 0 are generated by the points $p \in S(K)$, and they are all homologous. A positive staircase complex can be represented by a diagram such as Figure \ref{fig:2,5}, where the elements of $S$ are represented by black dots and the elements of $R$ are represented by the white dots.

We remark that for any positive L-space knot $K \in \mathcal{K}_L$, $CFK^\infty (K)$ is determined uniquely by its Alexander polynomial,  $\Delta_K(t)$, up to chain homotopy \cite{ozsvath2005knot}\cite{Krcatovich2018}. Because the staircase complex is determined by the cycles (since the vertex between two cycle steps $(i_n,j_n)$ and $(i_{n+1},j_{n+1})$ must be $(i_{n+1},j_n)$), we have $$\Delta_{K_a} (t)\neq \Delta_{K_b}(t) \implies S(K_a) \neq S(K_b).$$

We define a lexicographic order on points in $\mathbb{Z}^2$ $>_{\text{lex}}$ given by $(i_1,j_1) >_\text{lex}(i_2,j_2)$ if and only if $i_1 > i_2$ or $(i_1 = i_2 \text{ and }j_1 >j_2)$. We utilize this to construct a strict total order $\succ$ on $\mathcal{K}_L$. Suppose $K_a,K_b \in \mathcal{K}_L$ with $\Delta_{K_a}(t) \neq \Delta_{K_b}(t)$. Their symmetric difference $S(K_a) \triangle S(K_b)$ must be finite and nonempty. Then define $p^*(K_a,K_b) \in \mathbb{Z}^2$ as the unique maximal point in $S(K_a) \triangle S(K_b)$ under $>_\text{lex}$ written as $$p^*(K_a,K_b) = (i^*,j^*) = \max_{>_\text{lex}} \big( S(K_a) \triangle S(K_b)\big).$$ We then define the binary relation $\succ$ on L-space knots with different Alexander polynomial by $$K_a \succ K_b \iff p^*(K_a,K_b) \in S(K_a).$$

\begin{prop}
    The relation $\succ$ defines a strict total order on any set of L-space knots with pairwise distinct Alexander polynomials.
\end{prop}

\begin{proof}
    We now verify the four axioms of a strict total order. 

    1. Irreflexivity: For any $K$, $S(K) \triangle S(K) = \emptyset$ meaning no maximal $p^*$ so $K \nsucc K$.
    
    2. Asymmetry: Suppose $K_a \succ K_b$. Then $p^* = \max_{>_\text{lex}}(S(K_a) \triangle S(K_b)) \in S(K_a)$. Because $p^* \in S(K_a) \triangle S(K_b)$, we must have $p^* \notin S(K_b)$. This means $p^*(K_b,K_a) = p^* \notin S(K_b)$, meaning that $K_b \nsucc K_a$.

    3. Totality: For any $K_a \neq K_b$ with $\Delta_{K_a} \neq \Delta_{K_b}$, $S(K_a) \neq S(K_b)$ meaning that $S(K_a) \triangle S(K_b) \neq \emptyset$. The set is finite, meaning that $p^*$ uniquely exists. Because $p^* \in S(K_a) \triangle S(K_b)$, $p^*$ must be either in $S(K_a)$ or $S(K_b)$ so either $K_a \succ K_b$ or $K_b \succ K_a$.

    4. Transitivity: Suppose we have $K_a \succ K_b$ and $K_b \succ K_c$. Allow $p_1 = p^*(K_a,K_b)$ and $p_2 = p^*(K_b,K_c)$. We divide this into three cases. Suppose first $p >_\text{lex} p_1$ with $p \notin S(K_a) \triangle S(K_b)$ and $p \notin S(K_b) \triangle S(K_c)$ meaning that $p \notin S(K_a) \triangle S(K_c)$. At $p_1, \text{ } p_1 \in S(K_a)$ and $p_1 \notin S(K_b)$. Since $p_1 >_{\text{lex}}p_2, \text{ }p^1 \notin S(K_b) \triangle S(K_c) $ so $p_1 \notin S(K_c)$. That means $p_1 \in S(K_a) \setminus S(K_c)$. Thus $p^*(K_a,K_c) = p_1 \in S(K_a)$ meaning that $K_a \succ K_c$.
    
    If $p_2 >_\text{lex} p_1$, we must have $p_2 \notin S(K_a) \triangle S(K_b)$. We also note since $K_b \succ K_c$ we have $p_2 \in S(K_b ) \setminus S(K_c)$ meaning $p_2 \in S(K_b)$, and because $p_2 \notin S(K_a) \triangle S(K_b)$, we must have $p_2 \in S(K_a)$ so $p_2 \in S(K_a) \setminus S(K_c)$. Thus $K_a \succ K_c$
    
    Finally, it is impossible for $p_1=p_2$ because $p_1 \notin S(K_b)$ but $p_2 \in S(K_b)$.

    Thus $\succ$ is a strict total order, and this means that every nonempty finite subset of distinct L-space knots must have a unique maximal element under $\succ$.
\end{proof}

\begin{prop} \label{5.1}
    $$\beta_K(t^1, a, t^2) = \min_{1 \le m \le k} \max \Big( f_{t^1}(i_m, j_m) - a, \;\; f_{t^2}(i_m, j_m) \Big)$$
\end{prop}

\begin{proof}

Let $C$ be a staircase complex. This means $C_0$ is spanned by Maslov grading 0 cycle generators $z_1,...,z_k$. Every Maslov grading 0 cycle can be written in the form $$z = \sum_{m=1}^k c_mz_m.$$ Since $\partial r_m = z_m + z_{m+1}$, a Maslov grading cycle $\sum c_m z_m$ represents a nonzero homology class exactly when we have $\sum_m c_m=1$ in $\mathbb{F}$. Additionally, $$f_t(z) = \max_{c_m \neq 0} f_t(z_m).$$ This means a nonzero cycle in the intersection must have filtration at least $$\min_m \max \{ f_{t^1}(z_m) - a, f_{t^2}(z_m)\}.$$ We get equality by taking the corresponding vertex. Thus the claim follows.
\end{proof}

\begin{lemma} \label{thm:Lemma 4}
    For positive L-space knots $K_1, K_2 \in \mathcal{K}_L$ and parameter $a \in \mathbb{R}$ $$\beta_{K_1 \# K_2}(t^1, a, t^2) = \min_{a_1 + a_2 = a} \left( \beta_{K_1}(t^1, a_1, t^2) + \beta_{K_2}(t^1, a_2, t^2) \right)$$ 
\end{lemma} 

\begin{proof}
  For any $X_1,X_2,Y_1,Y_2 \in \mathbb{R}$ and any $a \in \mathbb{R}$, we want to show $$\max(X_1 + Y_1 - a, \, X_2 + Y_2) = \min_{a_1 + a_2 = a} \left( \max(X_1 - a_1, X_2) + \max(Y_1 - a_2, Y_2) \right).$$ Here we take the minimum over all $a_1,a_2 \in \mathbb{R}$ with $a_1+a_2 = a.$

  We let $\Delta_1 = X_1 - X_2$ and $\Delta_2 = Y_1 - Y_2$. Define $h_1(a_1) = \max(X_1 - a_1, X_2)$ and $h_2(a_2) = \max(Y_1 - a_2, Y_2)$. We note that $h_1(a_1) \ge X_1 - a_1$ and $h_1(a_1) \ge X_2$ for all $a_1 \in \mathbb{R}$, with $h_1(a_1) = X_1 - a_1$ when $a_1 \le \Delta_1$, and $h_1(a_1) = X_2$ when $a_1 \ge \Delta_1$. This is true likewise for $h_2$.

  We now divide our proof into two cases $a \le \Delta_1 + \Delta_2$ and $a > \Delta_1 + \Delta_2$. 

  In the case where $a \le \Delta_1 + \Delta_2$, we have $\max(X_1 + Y_1 - a, X_2 + Y_2) = X_1 + Y_1 - a$. Suppose we set our parameters $$a_1 = \Delta_1 - \frac{(\Delta_1 + \Delta_2) - a}{2}, \quad a_2 = \Delta_2 - \frac{(\Delta_1 + \Delta_2) - a}{2}.$$ Summing these gives $a_1 + a_2 = \Delta_1 + \Delta_2 - ((\Delta_1 + \Delta_2) - a) = a$. Since $(\Delta_1 + \Delta_2) - a \ge 0$, we obtain $a_1 \le \Delta_1$ and $a_2 \le \Delta_2$. Then $$h_1(a_1) + h_2(a_2) = (X_1 - a_1) + (Y_1 - a_2) = X_1 + Y_1 - (a_1 + a_2) = X_1 + Y_1 - a.$$ Because $h_1(a_1) + h_2(a_2) \ge (X_1 - a_1) + (Y_1 - a_2) = X_1 + Y_1 - a$ everywhere, this gives us the global minimum.

  In the area where $a > \Delta_1 + \Delta_2$, $\max(X_1 + Y_1 - a, X_2 + Y_2) = X_2 + Y_2$. Here we set $$a_1 = \Delta_1 + \frac{a - (\Delta_1 + \Delta_2)}{2}, \quad a_2 = \Delta_2 + \frac{a - (\Delta_1 + \Delta_2)}{2}$$ and again note $a_1+a_2=a$. Since $a - (\Delta_1 + \Delta_2) > 0$, we get $a_1 > \Delta_1$ and $a_2 > \Delta_2$. We also obtain $$h_1(a_1) + h_2(a_2) = X_2 + Y_2.$$ Finally, since $h_1(a_1) \ge X_2$ and $h_2(a_2) \ge Y_2$ everywhere, $h_1(a_1) + h_2(a_2) \ge X_2 + Y_2$ for all $a_1, a_2$, giving us the global minimum. 

  Let $C_1$ and $C_2$ be the positive staircase complexes associated to $K_1$ and $K_2$. Because they are filtered chain homotopy equivalent to $CFK^\infty (K_1)$ and $CFK^\infty(K_2)$ respectively, we have that their tensor $$C_1 \otimes_{F[U,U^{-1}]}C_2$$ is a filtered chain model for $CFK^\infty (K_1 \# K_2)$. Then by Lemma \ref{thm:concordance invaraint}, we must simply compute $\beta$ on the tensor product. 

  We write the Maslov grading zero cycle generators of $C_1$ as $x_1,...,x_k$ and those of $C_2$ as $y_1,..,y_\ell$. Every Maslov grading zero cycle which represents the nonzero $H_0$-class can be replaced without increasing either filtration by a linear combination of the tensors $x_m \otimes y_p$ for $1 \leq m \leq k$ and $1 \leq p \leq \ell$. Then one such tensor realizes the minimum. Since filtration is additive $f_t(x_m \otimes y_p) = f_t(x_m) + f_t(y_p)$. That means every Maslov grading zero cycle in the tensor can be written as $z = \sum_{m,p}c_{m,p} (x_m \otimes y_m)$ for $c_{m,p} \in \mathbb{F}.$ Because all our $x_m$ represent the same nonzero class in $H_0(C_1)$ and all of our $y_p$ represent the same nonzero class in $H_0(C_2),$ every $x_m \otimes y_p$ must represent the same generator of $H_0(C_1 \otimes C_2) \cong \mathbb{F}.$ That means that $z$ represents a nonzero homology class exactly when $\sum_{m,p} c_{m,p}=1$ in $\mathbb{F}.$ 

  We also have $f_t(z) = \max_{c_{m,p}\neq 0} (f_t(x_m) + f_t(y_p))$. This means that if $z$ represents a nonzero homology class, then the $x_m \otimes y_p$ appearing in $z$ has a filtration which is no greater than $z$. Thus, the minimum $\beta$ can be achieved by $x_m \otimes y_p$. This gives us $$\beta_{K_1 \# K_2}(t^1,a,t^2) = \min_{m,p} \max (f_{t^1}(x_m) + f_{t^1}(y_p) -a, f_{t^2}(x_m)+ f_{t_2}(y_p))$$ due to Proposition \ref{5.1}. Due to the identity proven above, we obtain $$\beta_{K_1 \# K_2}(t^1,a,t^2)=\min_{m,p}\min_{a_1+a_2=a} [\max(f_{t^1}(x_m)-a_1,f_{t^2}(x_m)+\max(f_{t^1}(y_p)-a_2, f_{t^2}(y_p))].$$ We now note that the minima over $m$ and $p$ are independent. Additionally by Propostion \ref{5.1}, the two mimnia are $\beta_{K_1}(t^1,a_1,t^2)$ and $\beta_{K_2}(t^1,a_2,t^2)$. Thus we have $$\beta_{K_1 \# K_2}(t^1, a, t^2) = \min_{a_1 + a_2 = a} \left( \beta_{K_1}(t^1, a_1, t^2) + \beta_{K_2}(t^1, a_2, t^2) \right).$$
  %We now apply this property to $$\beta_K(t^1, a, t^2) = \min_{1 \le m \le k} \max \Big( f_{t^1}(i_m, j_m) - a, \;\; f_{t^2}(i_m, j_m) \Big)$$ proven in Proposition 5.1. $$\beta_{K_1 \# K_2}(t^1, a, t^2) = \min_{m, p} \min_{a_1 + a_2 = a} \left( \max(f_{t^1}(x_m) - a_1, f_{t^2}(x_m)) + \max(f_{t^1}(y_p) - a_2, f_{t^2}(y_p)) \right)$$
%$$\implies \beta_{K_1 \# K_2}(t^1, a, t^2) = \min_{a_1 + a_2 = a} \left( \min_m \max(f_{t^1}(x_m) - a_1, f_{t^2}(x_m)) + \min_p \max(f_{t^1}(y_p) - a_2, f_{t^2}(y_p)) \right)$$

%This means $$\beta_{K_1 \# K_2}(t^1, a, t^2) = \min_{a_1 + a_2 = a} \left( \beta_{K_1}(t^1, a_1, t^2) + \beta_{K_2}(t^1, a_2, t^2) \right).$$

\end{proof}

\section{Example of $\beta$ }

We now return to the example of $T(2,5)$ to see the necessity of the $\beta$ invariant in distinguishing L-space knots. First, let us consider the genus 2 knot seen in Figure $\ref{fig:3}.$ 

\begin{figure}
    \centering
    \includegraphics[width=0.3
    \linewidth]{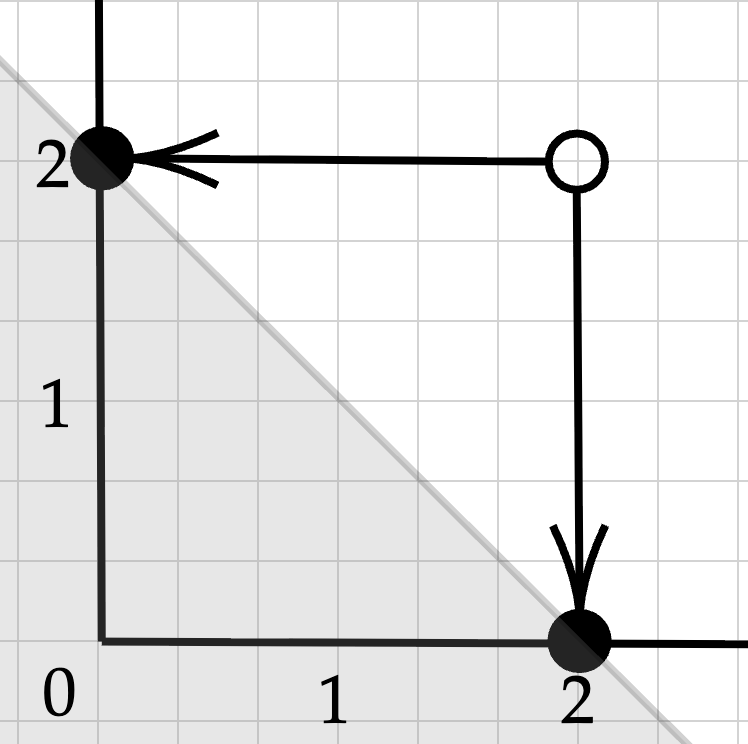}
    \caption{$CFK^\infty$ diagram of Genus 2, we call this $A$}
    \label{fig:3}
\end{figure}

Here, the filtration vertex which generates the minimal $s$ from time $0$ to $1$ is $(0,2)$, and from $1$ to $2$ it is $(2,0)$. Thus 
\[
\gamma(t) = \begin{cases} 
      t & 0\leq t \leq 1 \\
2-t &  1 < t \leq 2
   \end{cases}
\]

which implies 

\[
\Upsilon(t) = \begin{cases} 
      -2t & 0\leq t \leq 1 \\
-4+2t &  1< t \leq 2
   \end{cases}
\]

Notice this is the exact same $\Upsilon$ as $T(2,5)$. Thus, $\Upsilon$ is unable to distinguish these two staircase $CFK^\infty$ complexes.

\begin{figure}
    \centering
    \includegraphics[width=0.3\linewidth]{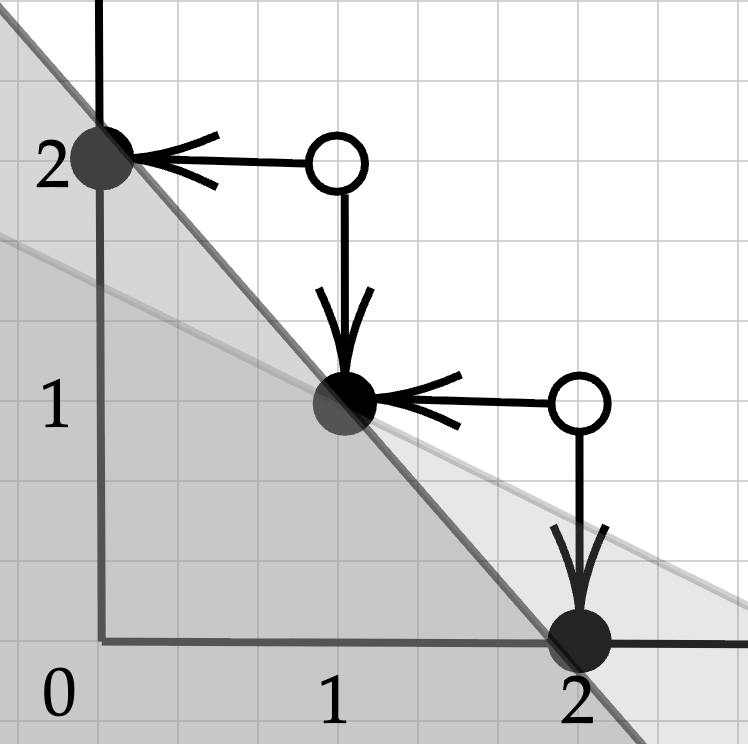}
    \caption{$CFK^\infty (T_{2,5})$ with two Filtration Lines}
    \label{fig:2,5-double}
\end{figure}

However, by using $\beta$, we see that for the two filtration lines corresponding to $t^1=0.9$  (the darker line) and $t^2=1.6$ (the lighter line), studying $\beta(0.9,a,1.6)$ allows us to detect all of the cycles of $T(2,5)$. 

Now we use the formula we derived for computing $s^2$ based on $a$ in Proposition \ref{5.1}.

Utilizing this formula, we obtain that for $T(2,5)$, $$\beta_{T(2,5)}(0.9,a,1.6)= \begin{cases}
0.9-a & \text{if } a < -0.7,\\
1.6  & \text{if } a \in [-0.7, -0.6),\\
1-a & \text{if } a \in [-0.6, 0),\\
1 & \text{if } a \in [0, 0.1),\\
1.1-a  & \text{if } a \in [0.1, 0.7),\\
0.4  & \text{if } a \geq 0.7.
\end{cases}$$
\begin{figure}
    \centering
    \includegraphics[width=0.5\linewidth]{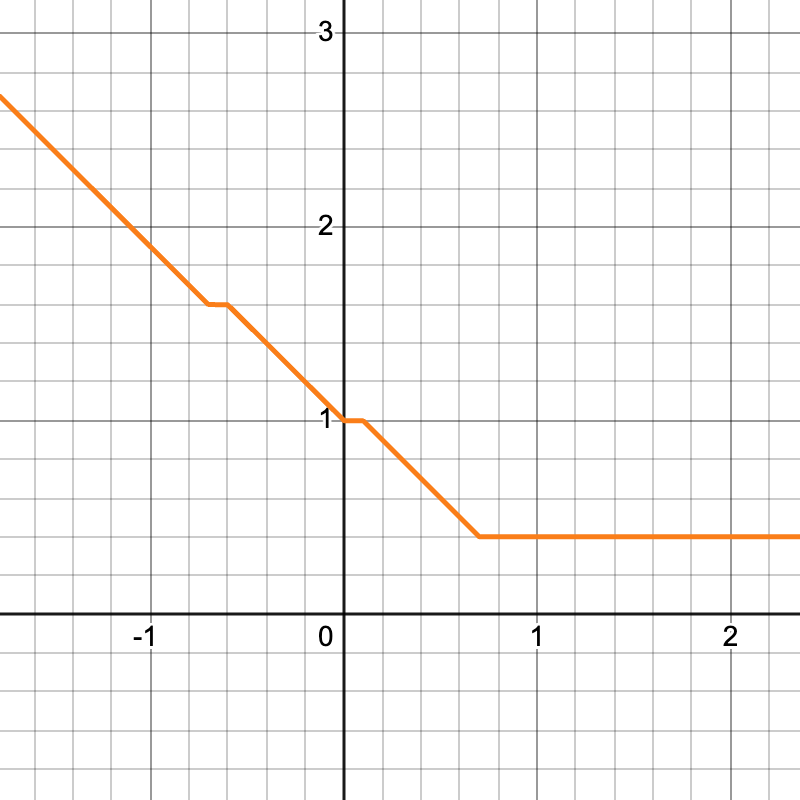}
,    \caption{$\beta_{T(2,5)}(0.9,a,1.6)$, generated using Desmos}
    \label{fig:placeholder}
\end{figure}

For the $CFK^\infty$ shown in Figure \ref{fig:3}, we instead obtain $$\beta_{A}(0.9,a,1.6)= \begin{cases}
0.9-a & \text{if } a \leq -0.7,\\
1.6  & \text{if } a \in [-0.7, -0.5),\\
1.1-a  & \text{if } a \in [-0.5, 0.7),\\
0.4  & \text{if } a \geq 0.7.
\end{cases}$$

Thus, while $\Upsilon$ is unable to distinguish $T(2,5)$ and $A$, $\beta$ is able to.

\begin{thm}
    Two positive L-space knots have the same $\beta$ if and only if they have the same Alexander polynomial. 
\end{thm}

\begin{proof}
We attempt to show that if $K_1$ and $K_2$ are two positive L-space knots, then $\beta_{K_1}(t^1,a,t^2)  =\beta_{K_2}(t^1,a,t^2) $ if and only if $S(K_1) = S(K_2)$, which indicates that they have the same Alexander polynomial.
If $S(K_1) = S(K_2)$, then by Proposition \ref{5.1}, $\beta(K_1) = \beta(K_2)$. 

We now prove the converse. If we take the filtration function $f_t(i, j) = (1 - \frac{t}{2})i + \frac{t}{2}j$ and study it at the limits $t^1 \to 0^+$ and $t^2 \to 2^-$ respectively we get $f_{t^1}(i, j) \to i \quad \text{and} \quad f_{t^2}(i, j) \to j$. When we pass to this limit in the formula $\beta_K(t^1, a, t^2) = \min_{1 \le m \le k} \max \Big( f_{t^1}(i_m, j_m) - a, \;\; f_{t^2}(i_m, j_m) \Big)$, we get $$\phi_K(a) := \lim_{\substack{t^1 \to 0^+ \\ t^2 \to 2^-}} \beta_K(t^1, a, t^2) = \min_{1 \le m \le k} \max \left( i_m - a, \; j_m \right).$$
If $\beta_{K_1} \equiv \beta_{K_2}$, then $\phi_{K_1}(a) = \phi_{K_2}(a)$ for all $a \in \mathbb{R}$. 

For all $(i_m, j_m) \in S(K)$, we define $h_m(a) = \max(i_m - a, j_m)$. For $a \le i_m - j_m$, $h_m(a) = i_m - a$ and for $a \ge i_m - j_m$, $h_m(a) = j_m$. The non-differentiable corner of $h_m(a)$ happens when $a_m = i_m - j_m$, where $h_m(a_m) = j_m$.

We then order the coordinates in $S(K)$ using our lexicographical order. Because $i_{m+1} > i_m$ and $j_{m+1} < j_m$, the transition points $a_m = i_m - j_m$ strictly increase
$$a_{m+1} - a_m = (i_{m+1} - i_m) + (j_m - j_{m+1}) \ge 1 + 1 = 2.$$ Thus, $a_1 < a_2 < \dots < a_k$. At $a = a_m$, we can now evaluate $h_{m'}(a_m)$ for all $m' \neq m$. For $m' < m$, because $a_m > a_{m'}$, $h_{m'}(a_m) = j_{m'}$ and because $j_{m'} > j_m$, $h_{m'}(a_m) > j_m$. For $m' > m$ because $a_m < a_{m'}$, $h_{m'}(a_m) = i_{m'} - a_m = i_{m'} - i_m + j_m$ and because $i_{m'} > i_m$, $h_{m'}(a_m) \ge j_m + 1 > j_m$. Thus we have $\phi_K(a_m) = \min_{m'} h_{m'}(a_m) = h_m(a_m) = j_m$. In a small neighborhood $(a_m - \epsilon, a_m + \epsilon)$, $h_m(a)$ achieves a strict minimum, so $\phi_K(a)$ changes its derivative from $-1$ to $0$ at $a_m$. That means every cycle vertex $(i_m, j_m) \in S(K)$ produces a point at which the one sided derivatives changes from $-1$ to $0$ in the function $\phi_K(a)$.

Now using $\phi_K(a)$, we give a method to find $S(K)$. We find every value $a^* \in \mathbb{R}$ where the derivative of $\phi_K(a)$ changes from $-1$ to 0. This set of corner points is $\{a_1, a_2, \dots, a_k\}$. Then for each $a_m$, find $y_m = \phi_K(a_m)$ which gives Alexander filtration $j_m=y_m$. Finally, compute the algebraic filtration $i_m = a_m + y_m$. Since $S(K) = \{(a_m + \phi_K(a_m), \phi_K(a_m)) \mid a_m \text{ is a corner of } \phi_K\}$, $S(K)$ is fully determined by $\phi_K(a)$. Thus if $\beta_{K_1} \equiv \beta_{K_2}$, then $\phi_{K_1} \equiv \phi_{K_2}$, which implies $S(K_1) = S(K_2)$.   
\end{proof}
Therefore, $\beta$ captures more information than $\Upsilon$.

% Adjust the point value as needed
%This immediately provides us with a corollary:

%\begin{corollary}
%Two L-space knots with different Alexander polynomials cannot be concordant.
%\end{corollary}

%Using $\beta$, we know they will have different $\beta$ invariants, and since $\beta$ is a concordance invariant, they cannot be concordant. Note that this is true even if both knots share the same $\Upsilon$ invariant. 

%We also note that if we have two positive L-space knots with different Alexander polynomials and $K_1 \# -K_2 \sim 0$, then this implies $K_1 \sim K_2$. However from Corollary 1, this is impossible, so as a result, for two L-space knots $K_1,K_2$ with different Alexander polynomial, $K_1 \# -K_2 \nsim 0$.

%In a recent paper \cite{Teragaito2025Hyperbolic}, Teragaito identifies infinitely many pairs of hyperbolic L-space knots $K_1$ and $K_2$ that have distinct Alexander polynomials but share the same $\Upsilon$. Thus while $\Upsilon(K_1 \# -K_2 ) = 0$, we know using $\beta$ that $K_1 \nsim K_2$ and thus $K_1 \# -K_2 \nsim 0$. 

%Thus, $\beta$ allows us to track the cycles of L-space knots, but it can only ever capture an L-space knot's amount of information. Thus, under $\beta$, for example, if we connect sum many positive L-space knots, it will only be able to capture the bottom L-space envelope. 

\section{Concordance Results}

In \cite{Teragaito2025Hyperbolic}, Teragaito constructs infinitely many pairs of L-space knots which have distinct Alexander polynomials, but share the same $\Upsilon$ invariant.  Teragaito's knots of type $K_{1,n}$ are defined using the diagrams in Figure \ref{fig:teresknots}. The Alexander polynomial of $K_{1,n}$ is 
    \begin{multline}
    \sum_{i=0}^{n} \left( t^{8n+12+4i} - t^{8n+11+4i} \right) + \left( t^{8n+9} - t^{8n+8} \right) + \\\sum_{i=0}^{n} \left( t^{4n+6+4i} - t^{4n+4+4i} \right) + \left( t^{4n+3} - t^{4n+1} \right) +  \sum_{i=0}^{n-1} \left( t^{4+4i} - t^{1+4i} \right) + 1
\end{multline} and the Alexander polynomial of $K_{2,n}$ is
\begin{multline}
\sum_{i=0}^{n} \left( t^{8n+12+4i} - t^{8n+11+4i} \right) + \left( t^{8n+9} - t^{8n+8} \right) + \\ \sum_{i=0}^{2n-1} \left( t^{4n+8+2i} - t^{4n+7+2i} \right) + \left( t^{4n+6} - t^{4n+4} \right) + \left( t^{4n+3} - t^{4n+1} \right) + \sum_{i=0}^{n-1} \left( t^{4+4i} - t^{1+4i} \right) + 1
\end{multline} In \cite{Teragaito2025Hyperbolic}, it is shown that $\Upsilon(K_{1,n}) = \Upsilon(K_{2,n})$.

\begin{proof}[Proof of Theorem \ref{thm:teragaito}]

We notice that knots $K_{1,n}$ and $K_{2,n}$ are of genus ${(8n+12+4n)/2=6n+6}$. The method for constructing the $CFK^\infty$ structure from the Alexander polynomial is by $$\frac{\Delta_K(t) - 1}{t - 1} = \sum_{b \in G} t^b.$$ Passing through a gap $b \in G$ then corresponds to moving down ($-j$ direction) by one vertical step in the $CFK^\infty$ staircase. Then the semigroup $ \mathbb{N}_0 \setminus G$ contains all non-gap exponents, found by the power series $$\frac{\Delta_K(t)}{1 - t} = \sum_{s \in S_K} t^s.$$ Passing through a non-gap corresponds to moving to the right ($+i$ direction) by one unit. By this reasoning, the staircases for $K_{1,n}$ and $K_{2,n}$ will be the same everywhere except where their Alexander polynomials differ. 

We notice that the summands which differ in the two Alexander polynomials are $\sum_{i=0}^{n} \left( t^{4n+6+4i} - t^{4n+4+4i}\right)$ in $K_{1,n}$ and $\sum_{i=0}^{2n-1} \left( t^{4n+8+2i} - t^{4n+7+2i} \right)$ in $K_{2,n}$. Dividing $(t^{4n+6+4i} - t^{4n+4+4i})$ by $(t-1)$ yields $t^{4n+4+4i} + t^{4n+5+4i}$. This creates blocks of 2 consecutive gaps followed by blocks of 2 consecutive non-gaps for $K_{1,n}$. Dividing $(t^{4n+8+2i} - t^{4n+7+2i})$ by $(t-1)$ yields $t^{4n+7+2i}$. This creates single isolated gaps alternating with single isolated non-gaps for $K_{2,n}.$

In Lemma 4.2 of \cite{Teragaito2025Hyperbolic}, we are provided with a function $f(x)$, whose Legendre transform is $\Upsilon$ for both $K_{1,n}$ and $K_{2,n}$. This method is developed in \cite{borodzik2018upsilon}. By taking the Legendre transforms, we obtain $$\Upsilon_n(t) = \begin{cases}  (-6n - 6)t & \text{for } 0 \le t < \frac{1}{2}, \\ (-2n - 6)t - 2n & \text{for } \frac{1}{2} \le t < \frac{2}{3}, \\ -2nt - 2n - 4 & \text{for } \frac{2}{3} \le t < 1, \\ 2nt - 6n - 4 & \text{for } 1 \le t <\frac{4}{3}, \\ (2n + 6)t - 6n - 12 & \text{for } \frac{4}{3} \le t < \frac{3}{2}, \\ (6n + 6)t - 12n - 12 & \text{for } \frac{3}{2} \le t \le 2. \end{cases}.$$

Suppose for contradiction $$K^{(1)}:= \#_n c_n K_{1,n} \sim K^{(2)}:= \#_m d_m K_{2,m}$$ for $c_n,d_m \in \mathbb{Z}_{\geq 0}$ and that both connect sums are nontrivial. Then by Lemma \ref{thm:concordance invaraint}, $\beta_{K^{(1)}}(t^1,a,t^2) = \beta_{K^{(2)}}(t^1,a,t^2)$ for all $t^1,t^2 \in (0,2)$ and $a \in \mathbb{R}$. Let us set $t_1 = 1/2$ and $t_2 = 1$. We compare $\beta_K(1/2,a,1)$ at parameters $a$ which we determine by the $f_{2,3}$ minimizing cycle vertices.

For any cycle vertex $x = (i,j)$, $f_{1/2}(x) = \frac{3}{4}i + \frac{1}{4}j$ and $f_1(x) = \frac{1}{2}i+\frac{1}{2}j.$ Thus $f_{1/2}(x) - f_{1}(x) = \frac{1}{4}(i-j)$. Additionally $f_{2/3} = \frac{2}{3}f_{1/2}+ \frac{1}{3}f_1.$ Note that the singularity of $\Upsilon_{K_{i,n}}$ at $t=2/3$ means the corresponding staircase segments have the cycles which lay on the $f_{2/3}$ minimizing face.

Now suppose $C$ is some connected sum of positive L-space knots, and let $$L_C = \min \{ f_{2/3}(x) :\text{ } x\text{ is a cycle representing the nonzero class} \}.$$ If $p= (i,j)$ is a $f_{2/3}$ minimizing cycle vertex, we set $a_p = \frac{1}{4}(i-j).$ Then since $f_{1/2}(x) - f_{1}(x) = \frac{1}{4}(i-j)$, $f_{1/2}(p) - a_p = f_1(p).$ Thus for any cycle vertex $x$, $\max \{f_{1/2} (x) - a_p, f_1(x)\} \geq \frac{2}{3}(f_{1/2}(x) -a_p) + \frac{1}{3}f_1(x) = f_{2/3}(x) - \frac{2a_p}{3} \geq L_C - \frac{2a_p}{3}. $ We note for $x=p$, we have equality because $f_{1/2}(p)-a_p = f_1(p) = L_C - \frac{2a_p}{3}$. Thus we have by Proposition \ref{5.1} that $$\beta_C(\frac{1}{2},a_p,1) = L_C - \frac{2a_p}{3}.$$ 

Conversely, if we suppose $\beta_C(\frac{1}{2},a_p,1) = L_C - \frac{2a_p}{3}$, the inequalities above give that any cycle which realizes the minimum satisfies $f_{2/3}(x) = L_C$. The equality holds in the maximum, as required by Proposition \ref{5.1}, so $f_{1/2}-a = f_1$ meaning $a = (i-j)/4$ for some $f_{2/3}$ minimizing cycle vertex $(i,j)$. Thus the set $$\big\{ a \in \mathbb{R}: \beta_C (\frac{1}{2},a,1) = L_C - \frac{2a}{3}\big\}$$ is exactly the set of $(i-j)/4$ which are obtained from the $f_{2/3}$ minimizing cycle vertices of C. 

We use an analogous characterization for $f_1$, where we replace $f_{2/3}$ by $f_1$ and set $t_1=1/2$ and $t_2 = 3/2$. Take $M_C$ to be the $f_1$ minimum. The same argument gives us that $\beta_C (1/2,a,3/2) = M_C - a/2$. This occurs exactly for the values $(i-j)/2$ coming from the $f_1$ minimizing cycle vertices of $C$. By the Lemma 4 connected sum formula for $\beta$, we have equality when \(a\) is the sum of the corresponding \((i-j)/2\) values from \(f_1\) minimizing vertices in the individual summands.

We now use the ordering $\succ$ defined above to find the first difference between two staircases. The staircases $K_{1,n}$ and $K_{2,n}$ agree until the region where their Alexander polynomials differ, and they agree again after this region. In the differing region, $K_{1,n}$ has 2 by 2 steps, whereas $K_{2,n}$ has 1 by 1 steps. 

For each summand, we let $p_r=(i_r,j_r)$ be the cycle vertex immediately before the first different step, with the choice of vertex made by $>_{\text{lex}}$. Then the next cycle vertex in the $K_{1,n}$ staircase is $p_r + (2,-2)$ and the next cycle vertex in the $K_{2,n}$ staircase is $p_r + (1,-1)$. 

The relevant step vectors in the $K_{1,n}$ staircase are $(1,-2), \text{ } (2,-2),\text{ }(2,-1)$, which come from the common 1 by 2 region, the differing 2 by 2 region, and the common 2 by 1 region. Thus, any displacement obtained by traversing a sequence of steps after $p_r$ is of the form $a(1,-2)+b(2,-2)+c(2,-1),$ for $ a,b,c\in\mathbb Z_{\geq0}$. If its displacement were $(1,-1)$, then its first coordinate would give $a+2b+2c=1.$ Because $a,b,c$ are nonnegative integers, this forces $a=1,$ and $ b=c=0$. The second coordinate is $-2$, not $-1$. Thus no cycle vertex of the $K_{1,n}$ staircase can occur at
displacement $(1,-1)$ from $p_r$.

%For every $K_{1,n}$, the relevant portion of the staircase, meaning the one that differs from $K_{2,n}$, has gap sequence consisting of blocks of two consecutive gaps followed by two consecutive non-gaps. Thus the consecutive cycle vertices in the $f_{2,3}$ minimizing part differ by $(2,-2)$. This means their $i-j$ values differ by 4. Thus all $i-j$ coming from $f_{2/3}$ minimizing cycle vertices of $K_{1,n}$ live in a single congruence class modulo 4. 

%For every $K_{2,n}$, the relevant portion of the staircase has gap sequence of isolated gaps followed by isolated non-gaps. Thus there are consecutive $f_{2/3}$ minimizing cycle vertices which differ by $(1,-1)$. Thus their $i-j$ values differ by 2 so the $f_{2/3}$ minimizing cycle vertices of $K_{2,n}$ give $i-j$ values in at least two distinct congruence classes modulo 4. 

%Under connect sum, the coordinates of tensor products add. Thus the $i-j$ values of $f_{2/3}$ minimizing cycle vertices must also add. Thus we have for $K^{(1)}$ all such $i-j$ values lie in one congruence class modulo 4 and for $K^{(2)}$ we have $f_{2/3}$ minimizing cycle vertices with $i-j$ values which differ by 2. 

We now utilize that the staircases of $K_{1,n}$ and $K_{2,n}$ agree outside the region where their Alexander polynomials differ. Particularly, we pick the first common point $p_r = (i_r,j_r)$ before the differing vertex of each summand where the two staircase patterns differ. In the $K_{1,n}$, the next cycle vertex after $p_r$ is $p_r + (2,-2)$ and in $K_{2,n}$, the next cycle vertex after $p_r$ is $p_r+(1,-1)$. Here we use the order defined above on points.

Now consider $K^{(1)}$ and $K^{(2)}$. In each summand, choose the corresponding vertex $p_r$ as above and let $p = \sum_r p_r = (I,J)$. Because the coordinates of tensor products add, $p$ is the corresponding cycle vertex in every connected sum. The two connected sums have identical staircase data up to $p$, because all of the corresponding summands agree before their first differing region. We note that these connected sums may not in themselves be represented as staircase complexes, as they may not be L-space. 

Because $p$ is the first point at which the differing region begins, the preceding cycle vertices lie above the relevant $f_{1}$ supporting face and thus do not contribute to the $f_{1}$ minimum. This means $p$ is an $f_{1}$ minimizing cycle vertex in both connect sums. 

At $p$, the $K^{(2)}$ staircase has a cycle vertex $q=p+(1,-1)$, but the $K^{(1)}$ staircase has its next cycle vertex at $p+(2,-2)$. By our previous step-vector argument, there is no cycle vertex of $K^{(1)}$ which is at displacement $(1,-1)$ from $p$.

Since $p=(I,J)$ is an $f_{1}$ minimizing cycle vertex, we can set $a_p = \frac{I-J}{2}$. The additional cycle vertex we have on the $K^{(2)}$ side gives us $a_q =\frac12\big((I+1)-(J-1)\big) =a_p+1.$ 

Because $K^{(1)}$ and $K^{(2)}$ are concordant, they share an $\Upsilon$ invariant. Particularly, $$ M_{K^{(1)}}=M_{K^{(2)}}=:M,
$$ since $\Upsilon_K(1)=-2M_K$.  Then by Lemma \ref{thm:concordance invaraint}, we have $$\beta_{K^{(1)}}\left(\frac12,a,\frac{3}{2}\right) = \beta_{K^{(2)}}\left(\frac12,a,\frac{3}{2}\right)$$ for every $a\in\mathbb R$.

Now we take $a = a_q$. Because $q$ is an $f_{1}$ minimizing cycle vertex of $K^{(2)}$, we have $$\beta_{K^{(2)}}\left(\frac12,a_q,\frac32\right) = M-\frac{a_q}{2}.$$

Now suppose that equality also holds for $K^{(1)}$. This would imply $K^{(1)}$ has an
$f_{1}$ minimizing cycle vertex $(i,j)$ such that $\frac{i-j}{2}=a_q,$ meaning $i-j=(I-J)+2.$
This equivalently means that relative to $p$, the corresponding cycle would occur at displacement $(1,-1)$ along the staircase. By the ordering and step-vector argument we have above, no such cycle exists in $K^{(1)}$.  Thus \[
\beta_{K^{(1)}}\left(\frac12,a_q,\frac32\right)
>
L-\frac{a_q}{2}
=
\beta_{K^{(2)}}\left(\frac12,a_q,\frac32\right),
\]
contradicting Lemma \ref{thm:concordance invaraint}. 

\end{proof}

%\pretzelSum*

%\begin{proof}
   % Due to work of Belousov \cite{belousov2025explicitformulasalexanderpolynomial}, we know the pretzel knots we care about have Alexander polynomial of the form $$\Delta_{P(-2, 3, n)}(t) = t^{n+3} - t^{n+2} + \left( \sum_{k=3}^{n} (-1)^{n-k} t^k \right) - t + 1.$$ Thus $P(t) = \frac{\Delta(t) - 1}{t - 1} = t^{n+2} + t^{n-1} + t^{n-3} + t^{n-5} + \dots + t^4 + t^2 + t$ and the set of gaps $G$ which are the exponents of $P(t)$) is explicitly $G = \{1, 2, 4, 6, 8, \dots, n-1, n+2\}$. Thus, using the logic of the previous proof, the first step is of the form right one down two, the last step is of the form right two down one, and all the other steps are of the form right one down one.

  %  For each of these pretzel knots, the steepest step is of the form right one down two. Thus, when we connect sum them, the connect sums of these right one down two steps will be outermost and visible to $\beta$. Additionally, upon connect summing $n$ such knots, we would obtain $n$ such outer steps. If we tensor every knot's vertex of the form $(1,\text{genus}-2)$, we get n steps of this form. $\beta$ can detect all these steps. Thus, iff $\#_n c_nK_n \sim \#_m d_mJ_m$, then the number of such steps must be the same since $\beta$ detects them all, meaning $\sum_n c_n = \sum_m d_m$.
%\end{proof}